\documentclass[a4paper,12pt]{article}
\usepackage{amssymb,amsmath,amsthm,latexsym}
\usepackage{amsfonts}
	\usepackage{amsfonts}
\usepackage{graphicx}
\usepackage[pdftex,bookmarks,colorlinks=false]{hyperref}
\usepackage{verbatim}
\usepackage{caption}
\usepackage{subcaption}

\newtheorem{theorem}{Theorem}[section]

\newtheorem{definition}[theorem]{Definition}

\newtheorem{problem}{Problem}
\newtheorem{proposition}[theorem]{Proposition}
\newtheorem{remark}[theorem]{Remark}

\title{\bf A Study on Arithmetic Integer Additive Set-Indexers of Graphs}
\author{{\bf N K Sudev \footnote{Department of Mathematics, Vidya Academy of Science \& Technology, Thalakkottukara, Thrissur - 680501, email: {\em sudevnk@gmail.com}}} and {\bf K A Germina\footnote{Department of Mathematics, School of Mathematical \& Physical Sciences, Central University of Kerala, Kasaragod, email:{\em srgerminaka@gmail.com}}}}
\date{}
\begin{document}
\maketitle

\begin{abstract}
A set-indexer of a graph $G$ is an injective set-valued function $f:V(G) \rightarrow2^{X}$ such that the function $f^{\oplus}:E(G)\rightarrow2^{X}-\{\emptyset\}$ defined by $f^{\oplus}(uv) = f(u ){\oplus} f(v)$ for every $uv{\in} E(G)$ is also injective, where $2^{X}$ is the set of all subsets of $X$ and $\oplus$ is the symmetric difference of sets. An integer additive set-indexer is defined as an injective function $f:V(G)\rightarrow 2^{\mathbb{N}_0}$ such that the induced function $f^+:E(G) \rightarrow 2^{\mathbb{N}_0}$ defined by $f^+ (uv) = f(u)+ f(v)$ is also injective. A graph $G$ which admits an IASI is called an IASI graph. An IASI $f$ is said to be a weak IASI if $|f^+(uv)|=max(|f(u)|,|f(v)|)$ and an IASI $f$ is said to be a strong IASI if $|f^+(uv)|=|f(u)| |f(v)|$ for all $u,v\in V(G)$. In this paper, we discuss about a special type of integer additive set-indexers called arithmetic integer additive set-indexer and establish some results on this type of integer additive set-indexers. We also check the admissibility of arithmetic integer additive set-indexer by certain graphs associated with a given graph.
\end{abstract}
\textbf{Key words}: Integer additive set-indexers, arithmetic integer additive set-indexers, deterministic index.\\
\textbf{AMS Subject Classification : 05C78}

\section{Preliminaries}

For all  terms and definitions, not defined in this paper, we refer to \cite{FH} and for more about graph labeling, we refer to \cite{JAG}. Unless mentioned otherwise, all graphs considered here are simple, finite and have no isolated vertices.

Let $\mathbb{N}_0$ denote the set of all non-negative integers. For all $A, B \subseteq \mathbb{N}_0$, the sum of these sets, denoted by  $A+B$, is defined by $A + B = \{a+b: a \in A, b \in B\}$.  In \cite{MBN}, the set $A+B$ defined above is termed as the {\em sumset} of the sets $A$ and $B$.

All sets mentioned in this paper are finite sets of non-negative integers. We denote the cardinality of a set $A$ by $|A|$.

An {\em integer additive set-indexer} (IASI, in short) is defined in \cite{GA} as an injective function $f:V(G)\rightarrow 2^{\mathbb{N}_0}$ such that the induced function $f^{+}:E(G) \rightarrow 2^{\mathbb{N}_0}$ defined by $f^{+} (uv) = f(u)+ f(v)$ is also injective.  A graph $G$ which admits an IASI is called an IASI graph. 
 
An IASI is said to be {\em $k$-uniform} if $|f^{+}(e)| = k$ for all $e\in E(G)$. That is, a connected graph $G$ is said to have a $k$-uniform IASI if all of its edges have the same set-indexing number $k$.

The cardinality of the labeling set of an element (vertex or edge) of a graph $G$ is called the {\em set-indexing number} of that element. 

The vertex set $V$ of a graph $G$ is defined to be {\em $l$-uniformly set-indexed}, if all the vertices of $G$ have the set-indexing number $l$.

 Two ordered pairs $(a,b)$ and $(c,d)$ in $A\times B$ {\em compatible} if $a+b=c+d$. If $(a,b)$ and $(c,d)$ are compatible, then we write $(a,b)\sim (c,d)$. Clearly, $\sim$ is an equivalence relation.

A {\em compatible class} of an ordered pair $(a,b)$ in $A\times B$ with respect to the integer $k=a+b$ is the subset of $A\times B$ defined by $\{(c,d)\in A\times B:(a,b)\sim (c,d)\}$ and is denoted by $[(a,b)]_k$ or $\mathsf{C}_k$. 

If $(a,b)$ is the only element in the compatibility class $[(a,b)]_k$, then it is called a {\em trivial class}. The compatibility classes which contain the maximum possible number of elements is called {\em saturated classes}. The compatibility class that contains maximum number of elements is called a {\em maximal compatibility class}.

The number of distinct compatibility classes in $A\times B$ is called the {\em compatibility index} of the pair of sets $A\times B$ and is denoted by $\mho_{(A,B)}$. Hence we have,

\begin{proposition}
If $A$ and $B$ are two non-empty sets of non-negative integers, then $|A+B|=\mho_{(A,B)}$.
\end{proposition}

\begin{proposition}\label{P-CardCC}
\cite{GS0} The maximum possible cardinality of a compatibility class in $A\times B$ is $n$, where $n=min(|A|,|B|)$. That is, the cardinality of a saturated class in $(A,B)$ is $min(|A|,|B|)$.
\end{proposition}

A {\em weak IASI} is defined in  \cite{GS1} as an IASI $f$ if for every $u,v\in V(G)$, $|f^{+}(uv)|=max(|f(u)|,|f(v)|)$. A graph which admits a weak IASI may be called a {\em weak IASI graph}. A weak  IASI is said to be {\em weakly uniform IASI} if $|f^{+}(uv)|=k$, for all $u,v\in V(G)$ and for some positive integer $k$. 

\begin{theorem}
\cite{GS0} For any two adjacent vertices $u$ and $v$ of $G$, if $|f^+(uv)|=min\{|f(u)|,|f(v)|\}$, then either $|f(u)|=1$ or $|f(v)|=1$. That is, if $f$ is a weak integer additive set-indexer of a graph $G$, then at least one end vertex of every edge of $G$ must have a singleton set-label.
\end{theorem}

A {\em strong IASI} is defined in \cite{GS2} as an IASI $f$ if $|f^{+}(uv)|=|f(u)| |f(v)|$ for all $u,v\in V(G)$. A graph which admits a  strong IASI may be called a {\em strong IASI graph}. A  strong  IASI is said to be  {\em strongly uniform IASI} if $|f^{+}(uv)|=k$, for all $u,v\in V(G)$ and for some positive integer $k$.

In this paper, we introduce a particular type of integer additive set-indexer called arithmetic integer additive set-indexer and establish some results on arithmetic integer additive set-indexers.

\section{Arithmetic Integer Additive Set-Indexers}

The integer additive set-indexers, under which the set-labels of the elements of a given graph $G$ following specific patterns, are of special interest. In this paper, we study the characteristics graphs, the elements of whose set-labels are in arithmetic progressions. Note that the elements in the set-labels of all elements of $G$ are in arithmetic progression, they must contain at least three elements. Hence, we have,

\begin{proposition}
An arithmetic integer additive set-indexer of a graph $G$ will never be a weak integer additive set-indexer of $G$.
\end{proposition}

By the term, an {\em arithmetically progressive set}, (AP-set, in short), we mean a set whose elements are in arithmetic progression. We call the common difference of the set-label of an element of a given graph, the {\em deterministic index} of that element.

\begin{definition}{\rm
Let $f:V(G)\to 2^{\mathbb{N}_0}$ be an IASI on $G$. For any vertex $v$ of $G$, if $f(v)$ is an AP-set, then $f$ is called a {\em vertex-arithmetic IASI} of $G$. A graph that admits a vertex-arithmetic IASI is called a {\em vertex-arithmetic IASI graph}.}
\end{definition}

\begin{definition}{\rm
For an IASI $f$ of $G$, if $f^+(e)$ is an AP-set, for all $e\in E(G)$, then $f$ is called an {\em edge-arithmetic IASI} of $G$. A graph that admits an edge-arithmetic IASI is called an {\em edge-arithmetic IASI graph}.}
\end{definition}

\begin{definition}{\rm
An IASI is said to be an {\em arithmetic integer additive set-indexer} if it is both vertex-arithmetic and edge-arithmetic. That is, an arithmetic IASI is an IASI $f$, under which the set-labels of all elements of a given graph $G$ are AP-sets. A graph that admits an arithmetic IASI is called an {\em arithmetic IASI graph}. }
\end{definition}

\begin{definition}{\rm
If all the set-labels of all vertices of a graph $G$ under an IASI,say $f$, are AP-sets and the set-labels of edges are not AP-sets,  then $f$ is called a {\em semi-arithmetic IASI}.}
\end{definition}

\begin{proposition}\label{P-AIASI-d}
Let $f$ be a vertex-arithmetic IASI defined on $G$. If the set-labels of vertices of $G$ are AP-sets with the same common difference $d$, then $f$ is also an edge-arithmetic IASI of $G$.
\end{proposition}
\begin{proof}
Let $u$ and $v$ be two adjacent vertices in $G$. Let $f(u)=\{a+id:0\le i \le m\}$ and $f(v)=\{b+jd:0\le j \le n$, where $a, b, m, n$ are non-negative integers. Then, $f^+(uv)=\{(a+b)+(i+j)d:0\le i \le m, 0\le j \le n\}$. That is, $f$ is edge-arithmetic.
\end{proof}

Let $G$ be a graph whose elements have different deterministic indices. Then, the following proposition establishes a necessary and sufficient condition for a vertex-arithmetic IASI of $G$ to be an arithmetic IASI of $G$. 

\begin{proposition}\label{P-AISAI-kd}
Let $f$ be a vertex arithmetic IASI of the graph $G$ such that deterministic indices of any two adjacent vertices in $G$ are distinct.  Then, $f$ is edge-arithmetic if and only if for any pair of adjacent vertices in $G$, the deterministic index of one is a positive integral multiple of the deterministic index of the other and this positive integer is less than or equal to the cardinality of the set-label of the latter.
\end{proposition}
\begin{proof}
Let $f:V(G)\to 2^{\mathbb{N}_0}$ be a vertex arithmetic IASI on $G$ such that $f(v_i)$ and $f(v_j)$ have the deterministic indices $d_i$ and $d_j$ respectively, where $d_i\neq d_j$.  Assume that $A=f(v_i)=\{a_r = a+rd_i:0 \le r <m\}$ and $B=f(v_j)=\{b_s=b+sd_i:0\le s <n\}$. Then, $|f(v_i)|=m$ and $|f(v_j)|=n$. Now, arrange the terms of $A+B=f^+(v_iv_j)$ in rows and columns as follows. For $b_s\in B, 0\le s<n$, arrange the terms of $A+b_s$ in $(s+1)$-th row in such a way that equal terms of different rows come in the same column of this arrangement.

Assume, without loss of generality, that $d_i=kd_j$ and $k\le m$. If $k<m$, then for any $a\in f(v_i)$ and $b\in f(v_j)$ we have $a+(b+d_j)= a+b+kd_i< a+b+md_i$. That is, a few final elements of each row of the above arrangement occur as the initial elements of the succeeding row (or rows) and the difference between two successive elements in each row is $d_i$ itself. If $k=m$, then the the difference between the final element of each row and the first element of the next row is $d_i$ and the difference between two consecutive elements in each row is $d_i$. Hence, if $k\le m$, then  $f^+(v_iv_j)$ is an AP-set with difference $d_i$.  Therefore, $f$ is an arithmetic IASI.

We prove the converse part in two cases.

\noindent {\em Case-1:} Let $d_j=kd_i$ where $k>m$. Then, the difference between two successive elements in each row is $d_i$, but the the difference between the final element of each row and the first element of the next row is $td_i$, where $t=k-m+1\neq 1$.  Hence, $f$ is not an arithmetic IASI.

\noindent {\em Case-2:}Assume that $d_j$ is not a multiple of $d_i$ (or $d_i$ is not a multiple of $d_j$). Without loss generality, let $d_i<d_j$. Then, by division algorithm, $d_j=pd_i+q, 0<q<d_i$. Then, the difference between any two consecutive terms in $f^+(v_iv_j)$ are not equal. Hence, $f$ is not an arithmetic IASI. This completes the proof.
\end{proof}

Due to Proposition \ref{P-AIASI-d} and Proposition \ref{P-AISAI-kd}, we have 

\begin{theorem}\label{T-AIASI-g}
A graph $G$ admits an arithmetic IASI graph $G$ if and only if for any two adjacent vertices in $G$, the deterministic index of one vertex is a positive integral multiple of the deterministic index of the other vertex and this positive integer is less than or equal to the cardinality of the set-label of the latter vertex.
\end{theorem}

\begin{proposition}\label{P-AIASI-GJ4}
Let $f$ be an IASI on $G$. Then, $f$ is edge-arithmetic implies it is vertex-arithmetic.
\end{proposition}
\begin{proof}
Let $f$ be an edge-arithmetic IASI on $G$. Then, $f^+(e)$ is an AP-set, for all $e\in E(G)$. If possible, let $f$ is not vertex-arithmetic. Hence, for two adjacent vertices $v_i$ and $v_j$ of $G$, let $A=f(v_i)$ and $B=f(v_j)$. Let $D_A$ and $D_B$ be the sets of all differences between the elements of $A$ and $B$ respectively. For any $d_i$ in $D_A$, there exists two elements $a_r$ and $b_s$ in $A+B$ such that $a_r-b_s=d_i$. Similarly, for any $d_j$ in $D_B$, there exists two elements $a_k$ and $b_l$ in $A+B$ such that $a_k-b_l=d_j$. Since $A+B$ is an AP-set, by Proposition \ref{P-AIASI-d}, $d_i=d_j$ or by Proposition \ref{P-AISAI-kd}, $d_j=kd_i, 1< k\le |f(v_i)|$. In both cases, $f(v_i)$ and $f(v_j)$ are AP-sets, a contradiction to the assumption that $f$ is not a vertex-arithmetic IASI.
Hence, every edge-arithmetic IASI is always a vertex-arithmetic IASI. 
\end{proof}

\begin{remark}{\rm
In view of Proposition \ref{P-AIASI-GJ4}, we note that all arithmetic (edge-arithmetic) IASI graphs are vertex arithmetic. But, a vertex-arithmetic IASI of graph $G$ need not be an edge-arithmetic IASI of $G$.}
\end{remark}

In the following theorem, we establish a relation between the deterministic indices of the elements of an arithmetic IASI graph $G$.

\begin{theorem}\label{T-AIASI-GJ5}
If $G$ is an arithmetic IASI graph, the greatest common divisor of the deterministic indices of vertices of $G$ and the greatest common divisor of the deterministic indices of the edges of $G$ are equal to the smallest among the deterministic indices of the vertices of $G$.
\end{theorem}
\begin{proof}
Let $f$ be an arithmetic IASI of $G$. Then, by Theorem \ref{T-AIASI-g}, for any two adjacent vertices $v_i$ and $v_j$ of $G$ with deterministic indices $d_i$ and $d_j$ respectively, either $d_i=d_j$, or if $d_j>d_i, d_j=k\, d_j$,  where $k$ is a positive integer such that $1<k\le |f(v_i)|$.

If the deterministic indices of the elements of $G$ are the same, the result is obvious. Hence, assume that for any two adjacent vertices $v_i$ and $v_j$ of $G$,  $d_j=k\, d_j, k\le |f(v_i)|$, where $d_i$ is the smallest among the deterministic indices of the vertices of $G$.  If $v_r$ is another vertex that is adjacent to $v_j$, then it has the deterministic index $d_r$ which is equal to either $d_i$ or $d_j$ or $l\,d_j$. In all the three cases, $d_r$ is a multiple of $d_i$. Hence, the g.c.d of $d_i, d_j, d_r$ is $d_i$. Proceeding like this, we have the g.c.d of the deterministic indices of the vertices of $G$ is $d_i$.

Also, by Theorem \ref{P-AISAI-kd}, the edge $u_iv_j$ has the deterministic index $d_i$. The edge $v_jv_k$ has the deterministic index $d_i$, if $d_k=d_i$, or $d_j$ in the other two cases. Proceeding like this, we observe that the g.c.d of the deterministic indices of the edges of $G$ is also $d_i$. This completes the proof.
\end{proof}

\begin{theorem}\label{T-AIASI1}
Let $G$ be a graph which admits an arithmetic IASI, say $f$ and let $d_i$ and $d_j$ be the deterministic indices of two adjacent vertices $v_i$ and $v_j$ in $G$, where $d_i<d_j$. Then, for some positive integer $1\le k\le |f(v_i)|$ , the edge $v_iv_j$ has the set-indexing number $|f(v_i)|+k(|f(v_j)|-1)$. 
\end{theorem}
\begin{proof}
Let $f$ be an arithmetic IASI defined on $G$. For any two vertices $v_i$ and $v_j$ of $G$, let $f(v_i)=\{a_i, a_i+d_i,a_i+2d_i,a_i+3d_i,\ldots,a_i+(m-1)d_i\}$ and let $f(v_j)=\{a_j, a_j+d_j,a_j+2d_j,a_j+3d_j,\ldots,a_j+(n-1)d_j\}$.

Let $d_j=k.d_i$, where $k$ is a positive integer such that $1\le k\le |f(v_i)|$. Then, $f(v_j)=\{a_j, a_j+kd_i,a_j+2kd_i,a_j+3kd_i,\ldots,a_j+(n-1)kd_i\}$. Therefore, $f^{+}(v_iv_j)=\{a_i+a_j, a_i+a_j+d_i,a_i+a_j+2d_i,\ldots,a_i+a_j+[(m-1)+k(n-1)]d_i\}$. That is, the set-indexing number of the edge $v_iv_j$ is $m+k(n-1)$. 
\end{proof}

\section{Arithmetic IASIs on Some Graph Classes}

\begin{theorem}
Every graph $G$ admits an arithmetic integer additive set-indexer.
\end{theorem}
\begin{proof}
Let $f$ be an IASI defined on a given graph $G$ under which all the vertices of $G$ are labeled by AP-sets in such a way that the common difference $f(v_i)$ is $d_i$, where $d_i$ is a positive integer. Let $v_1$ be an arbitrary vertex of $G$ and let $d_1$ be any positive integer. Label $v_1$ by an AP-set with common difference $d_1$. Let $v_2$ be a vertex of $G$ adjacent to $v_1$. Label this vertex by an AP-set with common difference $d_2=k_1\,d_1, 1\le k_1 \le |f(v_1)|$. Let $v_3$ be a vertex of $G$ adjacent to $v_2$. Label this vertex by an AP-set with common difference $d_3=k_2\,d_2, 1\le k_1 \le |f(v_1)|$. If $v_3$ is adjacent to $v_1$, then $d_3=k_3\,d_1, 1\le k_3 \le min(|f(v_1)|,\,|f(v_2)|)$. Label all vertices of $G$ in this manner. Then, by Theorem, the set-label of every edge of $G$ is also an AP-set.
Hence, $f$ is arithmetic IASI of $G$.
\end{proof}

The following theorem establishes the necessary and sufficient condition for a complete graph to admit an arithmetic IASI.

\begin{theorem}\label{T-AIASI-Kn}
A complete graph $K_n$ admits an arithmetic IASI if and only if its vertex set can be partitioned in to at most two sets such that every vertex in the same partition has the same deterministic index.
\end{theorem}
\begin{proof}
Let $K_n$ admits an arithmetic IASI. Let $d>0$ be the minimum of the deterministic indices of the vertices in $V(G)$. Let $V_1$ be the set of all vertices in $V(G)$ that have the deterministic index $d$ and let $V_2=V-V_1$. Since every vertex of $V_2$ is adjacent to all the vertices in $V_1$, by Theorem \ref{T-AIASI-g}, all the vertices in $V_2$ must have the deterministic index $k\,d$, where $k$ is a positive  by an AP-set with common difference $k\,d$, where $k$ is a positive integer which is less than or equal to the minimum value of the cardinalities the elements in $V_1$. If $V_1=V$, then $V_2=\emptyset$. That is, $V(G)$ has at most two partitions such that all vertices in the same partition have the same deterministic index.

If all the vertices in $V(G)$ have the same deterministic index, then by Proposition \ref{P-AIASI-d}, $K_n$ admits an arithmetic IASI. Hence, assume that not all vertices of $K_n$ have the same deterministic index. Then, $V(G)$ has two partitions, say $V_1=\{u_1,u_2,u_3,\ldots,u_r\}$ and $V_2=\{v_1,v_2,v_3,\ldots,v_l\}$, where $r+l=n$, such that all the vertices in each partition have the same deterministic index. Label all the vertices in $V_1$ by distinct AP-sets with the same common difference, $d>0$ and label all vertices in $V_2$ by distinct AP-sets with the same common difference $k\,d$, where $k$ is a positive integer such that $1\le k \le min\{|f(u_1)|,|f(u_2)|,|f(u_3)|,\ldots, |f(u_r)|\}$. Hence, by Theorem \ref{T-AIASI-g}, this labeling is an arithmetic IASI for $K_n$.  
\end{proof}

\begin{proposition}
If a graph $G$ admits an arithmetic integer additive set-indexer, then any non-trivial subgraph of $G$ also admits an arithmetic integer additive set-indexer.
\end{proposition}
\begin{proof}
let $f$ be an arithmetic IASI on $G$ and let $H\subset G$. The proof follows from the fact that the restriction $f|_H$ of $f$ to the subgraph $H$ is an arithmetic IASI on $H$.
\end{proof}

The following theorems establishes the admissibility of an arithmetic IASI by some graphs associated to a given arithmetic IASI graph.

\begin{definition}{\rm
\cite{FH} By {\em edge contraction operation} in $G$, we mean an edge, say $e$, is removed and its two incident vertices, $u$ and $v$, are merged into a new vertex $w$, where the edges incident to $w$ each correspond to an edge incident to either $u$ or $v$.}
\end{definition}

\begin{theorem}
Let $G$ be an arithmetic  IASI graph and let $e$ be an edge of $G$. Then, $G\circ e$ admits an arithmetic  IASI.
\end{theorem}
\begin{proof}
Let $G$ admits a weak IASI. Let $e$ be an edge in $E(G)$. Let $e=uv$ be an arbitrary edge of $G$. Let $d_i$ and $d_j$ be the deterministic number of $u$ and $v$ respectively. Then, by Theorem \ref{T-AIASI-g}, either $d_i=d_j$ or if $d_j>d_j$, $d_j=k\,d_j, 1\le k \le |f(u)|$. $G\circ e$ is the graph obtained from $G$ by deleting $e$ of $G$ and identifying $u$ and $v$ to get anew vertex, say $w$. Label the $w$, by the set-label of the deleted edge $e$. Then, $w$ has the deterministic number $d_i$ and all elements in $G\circ e$ are AP-sets. Hence, $G\circ e$ is a isoarithmetic IASI graph. 
\end{proof}

\begin{definition}{\rm
\cite{KDJ} Let $G$ be a connected graph and let $v$ be a vertex of $G$ with $d(v)=2$. Then, $v$ is adjacent to two vertices $u$ and $w$ in $G$. If $u$ and $v$ are non-adjacent vertices in $G$, then delete $v$ from $G$ and add the edge $uw$ to $G-\{v\}$. This operation is known as an {\em elementary topological reduction} on $G$.}
\end{definition}

\begin{theorem}
Let $G$ be a graph which admits an arithmetic IASI. Then, any graph $G'$, obtained by applying finite number of elementary topological reductions on $G$, also admits an arithmetic IASI. 
\end{theorem}
\begin{proof}
Let $G$ be a graph which admits an arithmetic IASI, say $f$. Let $v$ be a vertex of $G$ with $d(v)=2$ and deterministic index $d_i$. Since $d(v)=2$, $v$ must be adjacent two vertices $u$ and $w$ in $G$. Let these vertices $u$ and $w$ are non-adjacent. Now remove the vertex $v$ from $G$ and introduce the edge $uw$ to $G-{v}$. Let $G'=(G-{v})\cup \{uw\}$. Let $f':V(G')\to 2^{\mathbb{N}_0}$ such that $f'(v)=f(v)~ \forall ~v\in V(G')$ and the associated function $f'^+:E(G')\to 2^{\mathbb{N}_0}$ and defined by 
\[ {f'}^+(e)= \left\{
\begin{array}{l l}	
f^+(e)& \quad \text{if $e\ne uw$}\\
f(u)+f(w)& \quad \text{if $e=uw$}
\end{array} \right.\]
Hence, $f'$ is an arithmetic IASI of $G'$. 
\end{proof}

\begin{definition}{\rm
\cite{RJT} A {\em subdivision} of a graph $G$ is the graph obtained by adding vertices of degree $2$ into its edges.}
\end{definition}

\begin{theorem}
A graph  subdivision $G^{\ast}$ of a given arithmetic IASI graph $G$ also admits arithmetic IASI.
\end{theorem}
\begin{proof}
Let $u$ and $v$ be two adjacent vertices in $G$. Since $G$ admits an arithmetic IASI, the deterministic indices $d_i$ and $d_j$ of $u$ and $v$ respectively are either equal or, if $d_j>d_i$, $d_j=k\,d_i, 1\le k \le |f(u)|$. Introduce a new vertex $w$ to the edge $uv$. Now, we have two new edges $uw$ and $vw$ in place of $uv$. Extend the set-labeling of $G$ by labeling the vertex $w$ by the same set-label of the edge $uv$. Then, the vertices $u$ and $w$ have the same deterministic indices $d$ and the deterministic index of $v$ is a positive integral multiple of the deterministic index of $w$, where this positive integer is clearly less than the cardinality of the labeling set of $w$. Hence, $G^{\ast}$ admits an arithmetic IASI. 
\end{proof}

\begin{definition}{\rm
\cite{DBW} For a given graph $G$, its line graph $L(G)$ is a graph such that  each vertex of $L(G)$ represents an edge of $G$ and two vertices of $L(G)$ are adjacent if and only if their corresponding edges in $G$ incident on a common vertex in $G$.}
\end{definition}

\begin{theorem}\label{AIASI-lg}
The line graph $L(G)$ of an arithmetic IASI graph $G$ admits an arithmetic IASI. Moreover, the function $f^+$ associated to $f$ in $G$ is an arithmetic IASI on $L(G)$.
\end{theorem}
\begin{proof}
Consider two adjacent edges $e_1=v_1v_2$ and $e_2=v_2v_3$ in $G$. Let $d_i$ be the deterministic index of the vertex $v_i$ in $G$. Since $G$ is arithmetic graph, $d_1=d_2$ or $d_2=k_1\,d_1, 1 \le k \le |f(v_1)|$. In both cases, deterministic number of the edge $v_1v_2$ is $d_1$ itself. Similarly, since $v_2$ is adjacent to $v_3$, either $d_3=d_1$ or $d_3=d_2$ or $d_3=k_2\,d_2$. In all these three cases, the deterministic number of the edge $e_2$ is $d_1$ or a positive integer multiple of $d_1$. Proceed like this until all the elements of $G$ are set-labeled. Hence, we have a set-labeling in which for every pair of adjacent edges in $G$, the deterministic indices are either equal or the deterministic index of one of these vertices is a positive integral multiple of the other. Hence, the deterministic indices of corresponding vertices in $L(G)$ are either equal or the deterministic index of one of these vertices is a positive integral multiple of the other. That is, $f^+$ defined in $G$ is an arithmetic IASI on $L(G)$.
\end{proof}

\begin{definition}{\rm
\cite{MB} The {\em total graph} of a graph $G$ is the graph, denoted by $T(G)$, is the graph having the property that a one-to one correspondence can be defined between its points and the elements (vertices and edges) of $G$ such that two points of $T(G)$ are adjacent if and only if the corresponding elements of $G$ are adjacent (either  if both elements are edges or if both elements are vertices) or they are incident (if one element is an edge and the other is a vertex). }
\end{definition} 

\begin{theorem}
The total graph $T(G)$ of an arithmetic IASI graph $G$ admits an arithmetic IASI.
\end{theorem}
\begin{proof}
Define a function $g:V(T(G))\to 2^{\mathbb{N}_0}$ as follows.
\[ g(v')= \left\{
\begin{array}{cc}
f(v) & \text{if $v'\in V(L(G))$ corresponds to $v\in V(G)$}\\
f^+(e) & \text{if $v'\in V(L(G))$ corresponds to $e\in E(G)$}
\end{array} \right. \]
Since the vertices in $T(G)$ corresponding to the vertices of $G$ preserve the same set-labelings of the corresponding vertices of $G$ and the vertices of $T(G)$ corresponding to the edges of $G$ preserve the same set-labeling of the corresponding  edges in $G$, these vertices  in $T(G)$ satisfy the conditions required for admitting an arithmetic IASI. Hence $g$ is arithmetic IASI on $T(G)$.
\end{proof}

\section{Conclusion}

In this paper, we have discussed some characteristics of graphs which admit a certain type of IASI called arithmetic IASI. We have formulated some conditions for some graph classes to admit arithmetic IASIs. Problems related to the characterisation of different arithmetic and semi-arithmetic IASI graphs are still open.

The following problems on arithmetic and semi-arithmetic IASI graphs, analogous to the results proved for isoarithmetic IASI graphs, are to be investigated.

\begin{problem}
Discuss the admissibility of certain operations and products graphs which admit different types arithmetic IASI graphs.
\end{problem}

\begin{problem}
Characterise the graphs which admit different arithmetic and semi-arithmetic IASIs.
\end{problem}

\begin{problem}
Discuss the existence and cardinality of saturated classes in the set-labels of the elements of given graphs which admit arithmetic and semi-arithmetic IASIs. 
\end{problem}

\begin{problem}
Characterise the graphs which admit uniform arithmetic and semi-arithmetic IASIs.
\end{problem}

\begin{problem}
Discuss the admissibility of certain graphs and graph classes associated to given arithmetic and semi-arithmetic graphs.
\end{problem}

The IASIs under which the vertices of a given graph are labeled by different standard sequences of non negative integers, are also worth studying.   The problems of establishing the necessary and sufficient conditions for various graphs and graph classes to have certain IASIs still remain unsettled. All these facts highlight a wide scope for further studies in this area.


\begin{thebibliography}{15}
\bibitem {A10} B D Acharya, (1990). {\em Arithmetic Graphs}, J. Graph Theory, {\bf 14}(3), 275-299. 
\bibitem {AGA} B D Acharya, K A Germina and T M K Anandavally, {\em Some New Perspective on Arithmetic Graphs} In {\bf Labeling of Discrete Structures and Applications}, (Eds.: B D Acharya, S Arumugam and A Rosa), Narosa Publishing House, New Delhi, (2008), 41-46.
\bibitem{AG} B D Acharya and K A Germina, {\em Strongly Indexable Graphs: Some New Perspectives}, Adv. Modeling and Optimisation, {\bf 15}(1)(2013), 3-22.
\bibitem {TMA} Tom M Apostol, {\bf Introduction to Analytic Number Theory}, Springer-Verlag, New York, (1989).
\bibitem {MB} M Behzad, (1969). {\em The Connectivity of Total Graphs}, Bull. Austral. Math.Soc.,{\bf 1}, 175-181.
\bibitem {BM1} J A Bondy and U S R Murty, (2008). {\bf Graph Theory}, Springer.
\bibitem {BLS} A Brandst\"{a}dt, V B Le and J P Spinard, (1999). {\bf Graph Classes:A Survey}, SIAM, Philadelphia.
\bibitem {DMB} D M Burton, {\bf Elementary Number Theory}, Tata McGraw-Hill Inc., New Delhi, (2007).
\bibitem {CZ} G Chartrand and P Zhang, (2005). {\bf Introduction to Graph Theory}, McGraw-Hill Inc.
\bibitem {ND} N Deo, (1974). {\bf Graph Theory with Applications to Engineering and Computer Science}, PHI Learning.
\bibitem {JAG} J A Gallian, (2011). {\em A Dynamic Survey of Graph Labelling}, The Electronic Journal of Combinatorics (DS 16).
\bibitem {GA} K A Germina and T M K Anandavally, (2012). {\em Integer Additive Set-Indexers of a Graph:Sum Square Graphs}, Journal of Combinatorics, Information and System Sciences, {\bf 37}(2-4), 345-358.
\bibitem {GS1} K A Germina, N K Sudev, (2013). {\em On Weakly Uniform Integer Additive Set-Indexers of Graphs}, Int. Math. Forum., {\bf 8}(37), 1827-1834.
\bibitem {GS2} K A Germina, N K Sudev, {\em Some New Results on Strong Integer Additive Set-Indexers}, Communicated.
\bibitem {GY} J. Gross, J. Yellen, {\bf Graph Theory and Its Applications}, CRC Press, (1999).
\bibitem {FH}  F Harary, (1969). {\bf Graph Theory}, Addison-Wesley Publishing Company Inc.
\bibitem {KDJ} K D Joshi, {\bf Applied Discrete Structures}, New Age International, (2003).
\bibitem {SMH} S M Hegde, (1989). {\em Numbered Graphs and Their Applications}, PhD Thesis, Delhi University.
\bibitem {MBN} M B Nathanson (1996). {\bf Additive Number Theory, Inverse Problems and Geometry of Sumsets}, Springer, New York.
\bibitem {GS0} N K Sudev and K A Germina, (2014). {\em A Note on Integer Additive Set-Indexers of Graphs}, to appear in Int. J. Math. Sci.\& Engg. Applications, {\bf 8}(2).
\bibitem {RJT} R J Trudeau, (1993). {\bf Introduction to Graph Theory}, Dover Pub., New York.
\bibitem {DBW} D B West, (2001). {\bf Introduction to Graph Theory}, Pearson Education Inc.
\end{thebibliography}
\end{document}